\documentclass[11pt,draft,reqno]{article}
\usepackage{amssymb}
\usepackage{amsmath}
\usepackage{amsthm}
\usepackage{palatino}

\newtheorem{theorem}{Theorem}[section]
\newtheorem{lemma}[theorem]{Lemma}

\newtheorem{corollary}[theorem]{Corollary}
\theoremstyle{definition}

\theoremstyle{remark}

\numberwithin{equation}{section}
\begin{document}






\begin{center}
{\bf A STRONG  APPROXIMATION OF SUB-FRACTIONAL BROWNIAN MOTION\\
 
BY MEANS OF TRANSPORT PROCESSES}
\\[1cm]
Johanna Garz\'on, Luis G. Gorostiza and Jorge A. Le\'on
\end{center}
\vglue 1cm
\noindent
{\bf Abstract.} Sub-fractional Brownian motion is a process analogous to fractional Brownian motion but without stationary increments. In \cite{GGL1} we proved a strong uniform approximation with a rate of convergence for fractional Brownian motion by means of transport processes. In this paper we prove a similar type of approximation for sub-fractional Brownian motion.
\vglue .5cm
\noindent
{\bf 1. Introduction}
\setcounter{section}{1}
\setcounter{equation}{0}
\vglue .25cm
Fractional Brownian motion (fBm) is well known and used in many areas of application (see \cite{N,ST} for background, and \cite{DOT} for some applications). It is a centered Gaussian process $W=(W(t))_{t\geq 0}$ with covariance function
$$
E(W(s)W(t))=\frac{1}{2}(s^{2H} +t^{2H}-|s-t|^{2H}),\quad s, t\geq 0,
$$
where $H\in (0, 1/2) \cup (1/2,1)$  (the case $H=1/2$ corresponds to ordinary Brownian motion). $H$ is called Hurst parameter. The main properties of fBm  are that it is a continuous centered Gaussian process which is self-similar, has stationary increments with long range dependence, and is neither a Markov process nor a semimartingale. Since it is not a semimartingale, it has been necessary to develop new theories of stochastic calculus for fBm, different from the classical It\^o calculus (see e.g. \cite{BHOZ, Mi, N, NT} and references therein). 

Sub-fractional Brownian motion (sfBm) is a process $S=(S(t))_{t\geq 0}$ that has the main properties of   fBm except  stationary increments, and its long range dependence decays  faster  than that of fBm. Its covariance function is
$$
E(S(s)S(t))=s^{2H}+t^{2H}-\frac{1}{2}\left[(s+t)^{2H}+|s-t|^{2H}\right],\quad s, t\geq 0,
$$
with parameter $H\in (0,1/2)\cup (1/2, 1)$ (the case $H=1/2$ also corresponds to ordinary Brownian motion). The main properties of sfBm were studied in \cite{BGT1}, where it was also shown that it arises from the occupation time fluctuation limit of a branching particle system with $H$ restricted to $(1/2,1)$. This process  appeared independently in a different context in \cite{DZ}.

The emergence of sfBm has motivated a series of papers where it arises in 
connection with several analogous but somewhat different branching 
particle systems, usually with $H\in (1/2, 1)$. It has been shown in 
\cite{BGT3} that it also comes out in a more natural way from a particle 
system without branching, and in \cite{BT} there is a different particle 
picture approach that yields sfBm with the full range of parameters $H\in 
(0,1)$. Other long range dependent Gaussian processes have been obtained 
which are related to particle systems. A reader interested in fBm and sfBm 
in connection with particle systems can find some results and references 
in \cite{BGT1, BGT2, LX}.

Some authors have studied further properties of sfBm for its own sake and  related stochastic calculus, and possible applications of sfBm have been proposed (see  \cite{BB, EN, HN, LLY, LYPW, LY, M, No, RT, SCY, SY, SZ, Sw, T1, T2, T3, T4, T5, T6, YS,  YSH}).

There are various ways of approximating fBm in distribution that can be used for simulation of paths. In \cite{GGL1} we obtained a strong approximation of fBm  with a rate of convergence by means of the Mandelbrot-van Ness representation of fBm and  a strong  approximation of Brownian motion with transport processes proved in \cite{GG}. This was  employed in \cite{GGL2} for a strong approximation of solutions of fractional stochastic differential equations with a rate of convergence, which may be used for simulation of solutions  (computational efficiency was not the objective). A strong approximation of the Rosenblatt process by means of transport processes with a rate of convergence has been obtained in \cite{GTT}.

Since sfBm has attracted interest recently, it seems worthwhile to provide a strong approximation for it by means of transport processes with a rate of convergence, analogously as was done for fBm in \cite{GGL1}. This can be achieved using the same approach of \cite{GGL1} with some technical  modifications and additional work. The aim of the present paper is to prove such a strong  approximation for sfBm, which moreover has the same rate of convergence as that of the transport approximation of fBm. The result is given in Corollary 2.3.

We end the Introduction by recalling the strong transport approximation of Brownian motion. For each $n=1,2,\ldots$, let $(Z^{(n)}(t))_{t\geq 0}$ be a process such that $Z^{(n)}(t)$ is the position on the real line at time $t$ of a particle moving as follows. It starts from $0$ with constant velocity $+n$ or $-n$, each with probability $1/2$. It continues for a random time $\tau_1$ which is exponentially distributed with parameter $n^2$, and at that time it switches from velocity $\pm n$ to $\mp n$ and continues that way for an additional independent random time $\tau_2-\tau_1$, which is again exponentially distributed with parameter $n^2$. At time $\tau_2$ it changes velocity as before, and so on. This process is called a (uniform) transport process.

\begin{theorem}
\label{T1.1} \cite{GG} There exist versions on the transport process 
$(Z^{(n)}(t))_{t\geq 0}$ on the same probability space as a Brownian motion  
$(B(t))_{t\geq 0}$ such that for each $q>0$, 
$$
P\biggl(\sup_{a\leq t \leq b}|B(t)-Z^{(n)}(t) |>Cn^{-1/2}(\log n)^{5/2}\biggr) =o (n^{-q}) \quad {\it as}\quad n\rightarrow \infty,
$$
where $C$ is a positive constant depending on $a,b$ and $q$.
\end{theorem}

See \cite{GG,GGL1} for background and  references.

\vglue 1cm
\noindent
{\bf 2. Approximation}
\label{S2}
\setcounter{section}{2}
\setcounter{equation}{0}
\setcounter{theorem}{0}
\vglue .5cm
A stochastic integral representation of sfBm $S$ with parameter $H$ is given by
\begin{equation}
\label{eqdefS}
    S(t)= C\int_{-\infty}^{\infty}[\left((t-s)^+\right)^{H-1/2}+((t+s)^-)^{H-1/2} - 2((-s)^+)^{H-1/2}]dB(s),
\end{equation}
where $C$ is a positive constant depending on $H$, and $B=(B(t))_{t\in \mathbb{R}}$ is Brownian motion on the whole real line (see \cite{BGT1}). Rewriting (\ref{eqdefS}), we have
\begin{equation}
\label{eqrepS1}
    S(t)= W(t) + Y(t),
\end{equation}
where $W$ is a  fBm with Hurst parameter $H$ and Mandelbrot-van Ness representation  
\begin{equation}
\label{eqdefW}
W(t)= C\biggl\{\int_{-\infty}^0[\left(t-s\right)^{H-1/2}-(-s)^{H-1/2}]dB(s)+ \int_0^t (t-s)^{H-1/2}dB(s)\biggr\},
\end{equation}
and the process $Y$ is defined by
\begin{equation}
\label{eqdefY}
Y(t)= C\biggl\{\int_{-\infty}^{-t}[\left(-t-s\right)^{H-1/2}-(-s)^{H-1/2}]dB(s)- \int_{-t}^0 (-s)^{H-1/2}dB(s)\biggr\}.
\end{equation}

Due to (\ref{eqrepS1})-(\ref{eqdefY}), the processes $S$ and $Y$ have common properties in general, in particular the same H\"older continuity.

We fix $T>0$ and $a<-T$, and we consider the following Brownian motions  constructed from $B$:
\vglue .25cm
\noindent
(1) $\left(B_1(s)\right)_{0 \leq s \leq T}$, the restriction of $B$ to the interval $\left[0, T\right]$.\\
(2) $\left(B_2(s)\right)_{a \leq s \leq 0}$, the restriction of $B$ to the interval $\left[a, 0\right]$.\\
(3)     $B_3(s)=
\begin{cases} sB(\frac{1}{s}) & \text{if}\  s\in 
\left[1/a, 0\right),\\
    0& \text{if}\  s=0. 
\end{cases}$
\vglue .5cm
By Theorem \ref{T1.1} there are three transport processes 
$$
	(Z_1^{(n)}(s))_{0 \leq s \leq T}, \ (Z_2^{(n)}(s))_{a \leq s \leq 0}, \ \  \text{and} \ \ \ (Z_3^{(n)}(s))_{1/a \leq s \leq 0},
$$ 
such that  for each $q>0$,
\begin{equation}
\label{eq98}
    P\left(\sup_{b_i\leq t \leq c_i}|B_i(t) - Z_i^{(n)}(t)|> C^{(i)} n^{-1/2}(\log n )^{5/2}\right)= o(n^{-q}) \ \ \ \text{as}  \ \ n\to \infty,
\end{equation}
where $b_i, c_i$, $i=1,2,3$,  are the endpoints of the corresponding intervals, and each $C^{(i)}$ is a positive constant depending on $b_i$, $c_i$ and $q$. Note that $Z^{(n)}_2$ and  $Z^{(n)}_3$ are constructed going backwards in time.

We now proceed similarly as in \cite{GGL1}. We define the functions
$$
f_t(s)= (t-s)^{H-1/2}-(-s)^{H-1/2}\  \ \ \  \text{for} \ \ \  s< 0\leq t \leq T,
$$
$$
g_t(s)=(t-s)^{H-1/2}	\ \ \ \text{for}\ \ \ 0<s < t \leq T,
$$
and for $0<\beta<1/2$, we put
\begin{equation}
\label{eqdefepsilon}
\varepsilon_n=-n^{-\beta/|H-1/2|}.
\end{equation}

There are  different approximations of $W$   for $H> 1/2$ and  for $H<1/2$. We fix $0<\beta <1/2$. For $H>1/2$ we define the  process $W^{(n)}_{\beta}=\left(W^{(n)}_{\beta}(t)\right)_{ t \in [0, T]}$ by 
\begin{align*}
W^{(n)}_{\beta}(t)&=C_H\biggl\{\int_0^t g_t(s)dZ_1^{(n)}(s) + \int_{a}^0 f_t(s)dZ_2^{(n)}(s)+f_t(a)Z_2^{(n)}(a)\bigg.\notag\\
&\hspace{2.5cm}+\bigg.\int_{1/a}^{0}\biggl(-\int_{1/a}^{s\wedge {\varepsilon}_n}\partial_sf_t\biggl(\frac{1}{v}\biggr)\frac{1}{v^3}dv\biggr)dZ_3^{(n)}(s)\biggr\},
\end{align*} 
and for $H<1/2$ we define the process $\hat{W}^{(n)}_{\beta}=\left(\hat{W}^{(n)}_{\beta}(t)\right)_{t \in [0, T]}$ by
\begin{align*}
&\hat{W}^{(n)}_{\beta}(t)=C_H\biggl\{\int_0^{(t+\varepsilon_n)\vee 0} g_t(s)dZ_1^{(n)}(s) +  \int_{(t+\varepsilon_n)\vee 0}^t g_t(\varepsilon_n+s)dZ_1^{(n)}(s)\bigg.\notag\\
&+\int_a^{\varepsilon_n} f_t(s)dZ_2^{(n)}(s)+f_t(a)Z_2^{(n)}(a)+\bigg.\int_{1/a}^{0}\biggl(-\int_{1/a}^{s}\partial_sf_t\biggl(\frac{1}{v}\biggr)\frac{1}{v^3}dv\biggr)dZ_3^{(n)}(s)\biggr\}.
\end{align*} 
We write $W^{(n)}=(W^{(n)}(t))_{t\in [0, T]}$, where 
\begin{equation}
	\label{eqdefwn}
	W^{(n)}=\begin{cases}
	W^{(n)}_{\beta} & \text{if} \ \  H>1/2,\\
	\hat{W}^{(n)}_{\beta} & \text{if} \ \  H<1/2.\\
	\end{cases}
\end{equation}
Note that $W^{(n)}$ is defined on the same probability space as the Brownian motion $B$ in (\ref{eqdefW}), and recall that it depends on $\beta$ through (\ref{eqdefepsilon}).

The following theorem gives the convergence and the rate of convergence of  $W^{(n)}$  to $W$. 

\begin{theorem} 
\label{teoh1} \cite{GGL1} Let $H\neq 1/2$ and let $W$ and $W^{(n)}$ be the processes defined by (\ref{eqdefW}) and (\ref{eqdefwn}), respectively. Then for each $q>0$ and each $\beta$ such that $0<\left|H-1/2\right|< \beta < 1/2$, there is a constant $C>0$ such that  
\begin{equation*}
        P\left(\sup_{0\leq t \leq T}\left|W(t) - W^{(n)}(t)\right|> Cn^{-(1/2-\beta)}(\log n )^{5/2}\right)=o(n^{-q}) \ \ \text{as} \ n\to \infty.
\end{equation*}
\end{theorem}

We define another function
\begin{equation}
\label{eqdeffunF}
F_t(s)= (-t-s)^{H-1/2}-(-s)^{H-1/2}\  \ \ \  \text{for} \ \ \  s< -t <0.
\end{equation}
In \cite{GGL1} $Z^{(n)}_2(s)$ was defined for $s\in [a, 0]$ and $Z^{(n)}_3 (s)$ was defined for $ s\in [1/a,0]$, where $a<0$ was arbitrary, but for the approximation of sfBm we need   $a<-T$ so that $F_t(s)$ is well behaved.

Now we define  approximating processes for $Y$ and $S$ in (\ref{eqrepS1}), 
again for a fixed $0<\beta <1/2$.

For $H>1/2$ we define the process $Y^{(n)}_{\beta}=(Y^{(n)}_{\beta}(t))_{t\in [0, T]}$ by
\begin{align}
\label{eqdefyn1}
Y^{(n)}_{\beta}(t)&=C\biggl\{-\int^0_{-t} (-s)^{H-1/2}dZ_2^{(n)}(s) + \int_{a}^{-t} F_t(s)dZ_2^{(n)}(s)+F_t(a)Z_2^{(n)}(a)\bigg.\notag\\
&\hspace{2.5cm}+\bigg.\int_{1/a}^{0}\biggl(-\int_{1/a}^{[{\varepsilon}_n\vee (1/a)]\wedge s }\partial_sF_t\biggl(\frac{1}{v}\biggr)\frac{1}{v^3}dv\biggr)dZ_3^{(n)}(s)\biggr\},
\end{align} 
and for $H<1/2$ we define the process $\hat{Y}^{(n)}_{\beta}=\left(\hat{Y}^{(n)}_{\beta}(t)\right)_{t \in [0, T]}$ by
\begin{align}
\label{eqdefyn2}
\hat{Y }^{(n)}_{\beta}(t)=C&\biggl\{-\int_{-t}^{\varepsilon_n\vee (-t)} (-s)^{H-1/2}dZ_2^{(n)}(s)\notag\\
& -  \int_{\varepsilon_n\vee(-t)}^0 (-s-\varepsilon_n)^{H-1/2}dZ_2^{(n)}(s)\bigg. +F_t(a)Z_2^{(n)}(a)\notag\\
&-\int_{1/a}^{0}\biggl(-\int_{1/a}^{s}\partial_sF_t\biggl(\frac{1}{v}\biggr)\frac{1}{v^3}dv\biggr)dZ_3^{(n)}(s)+ \int_{a}^{a\vee(-t+\varepsilon_n)}F_t(s)dZ_2^{(n)}(s)\notag\\
&+\bigg. I_{\left\{-\varepsilon_n\leq t \right\}}\int_{a\vee(-t+\varepsilon_n)}^{-t}F_{t+\varepsilon_n}(s)dZ^{(n)}_2(s)\biggr\}.
\end{align} 
We write $Y^{(n)}=(Y^n(t))_{t\in [0, T]}$, where 
\begin{equation}
	\label{eqdefyn}
	Y^{(n)}=\begin{cases}
	Y^{(n)}_{\beta} & \text{if} \ \  H>1/2,\\
	\hat{Y}^{(n)}_{\beta} & \text{if} \ \  H<1/2,\\
	\end{cases}
\end{equation}
(note that $Y^{(n)}$ involves only $Z^{(n)}_2$ and $Z^{(n)}_3$), and we define 
\begin{align}
\label{eqdefsn}
S^{(n)}(t)&= W^{(n)}(t)+ Y^{(n)}(t),
\end{align} 
with $W^{(n)}$ as in (\ref{eqdefwn}).

The following theorem gives the convergence and the rate of convergence of  $Y^{(n)}$  to $Y$. 

\begin{theorem} 
\label{teoaproxy}
Let $H\neq  1/2$ and let $Y$ and $Y^{(n)}$ be the processes defined by (\ref{eqdefY}) and (\ref{eqdefyn}), respectively. Then for each $q>0$ and each $\beta$ such that 
$0<\left|H-1/2\right|< \beta < 1/2$, there is a constant $C>0$ such that  
$$
        P\left(\sup_{0\leq t \leq T}\left|Y(t) - Y^{(n)}(t)\right|> Cn^{-(1/2-\beta)}(\log n )^{5/2}\right)=o(n^{-q}) \ \ \text{as} \ n\to \infty.
$$
\end{theorem}

From Theorems \ref{teoh1} and \ref{teoaproxy} we have the following result.

\begin{corollary} 
\label{teoaproxs}
Let $S$ and $S^{(n)}$ be the processes defined by (\ref{eqdefS}) and (\ref{eqdefsn}), respectively. Then  for each $q>0$ and each $\beta$ such that 
$0<\left|H-1/2\right|< \beta < 1/2$, there is a constant $C>0$ such that  
$$
        P\left(\sup_{0\leq t \leq T}\left|S(t) - S^{(n)}(t)\right|> Cn^{-(1/2-\beta)}(\log n )^{5/2}\right)=o(n^{-q}) \ \ \text{as} \ n\to \infty.
$$
\end{corollary}
Note that the approximation becomes better when $H$ approaches $1/2$.
\vglue .5cm
\noindent
{\bf Remark 2.4}
\label{R2.5'}  The reason that the rates of convergence for $W$ and $Y$ are the same is that the integral representations of $W$ and $Y$, 
(\ref{eqdefW}) and (\ref{eqdefY}), have similar kernels, and the approximations depend basically on the rate of the transport approximation for Brownian motion and on the H\"older continuity of Brownian motion. Equation (\ref{eqrepS1}) is a decomposition of sfBm $S$ as a sum of a fBm $W$ and a process $Y$, which holds everywhere on the sample space, and $W$ and $Y$ are dependent (but the dependence does not play a role in the proofs). In \cite{BB} (which contains an approximation of sfBm in law) and \cite{RT}, for the case $H<1/2$, sfBm has a decomposition with equality in law as the sum of a fBm and a  process of the form 
$$
\int^\infty_0 (1-e^{-rt})r^{-(1+2H)/2}dB_1(r), \quad t\geq 0,
$$
where $B_1$ is a Brownian motion. This kind of process was introduced in \cite{LN}. In that decomposition the Brownian motions $B$ and $B_1$ are independent. 
That representation could be used for proving an approximation of sfBm with transport processes in the case $H<1/2$, but it would require another independent set of transport processes to approximate $B_1$. We stress that our approximation is strong and holds for all $H$.

\vglue 1cm
\noindent
{\bf 3. Proofs}
\setcounter{section}{3}
\setcounter{equation}{0}
\setcounter{theorem}{0}
\vglue .5cm
The proofs are based on a series of lemmas.

\begin{lemma}
\label{lemprop}
For each fixed $t>0$, the function $F_t$ defined by (\ref{eqdeffunF}) has the following properties:


\noindent
(1)
\begin{equation}
\label{eqpropF2}
|\partial_sF_t(s)|\leq |H-1/2|t(3/2-H)(-t-s)^{H-5/2}, \ \ s\leq -t.
\end{equation}
(2)
\begin{equation}
\label{eqpropF3}
\int_{-\infty}^{a} |\partial_sF_t(s)|(-s)^{1/2+\gamma}ds< \infty \ \text{for each} \  0<\gamma <(1-H)\wedge (1/2).
\end{equation}
(3)
\begin{equation}
\label{eqpropF4}
\lim_{b\to -\infty}F_t(b)B(b)=0 \ \text{a.s.}
\end{equation}
(4)
\begin{equation}
\label{eq:3.4}
 \int_{-\infty}^{a} F_t(s)dB(s)=F_t(a)B_2(a)-\int_{1/a}^0\partial_sF_t\left(\frac{1}{v}\right)\frac{1}{v^3}B_3(v)dv.
\end{equation}
\end{lemma}

\begin{proof}

\noindent  
(1)  
 $$\partial_s F_t(s)=(H-1/2)[(-s)^{H-3/2}-(-t-s)^{H-3/2}].$$ 
Taking   $g(x)=x^{H-3/2}$, \  $g'(x)=(H-3/2)x^{H-5/2}$, \  $x\in [-t-s, -s]$. By the mean value theorem, for some $r\in [-t-s, -s]$,
$$|(-s)^{H-3/2}-(-t-s)^{H-3/2}|=|-g'(r)(-t-s+s)|=t(3/2-H)r^{H-5/2}\leq t(3/2-H)(-t-s)^{H-5/2}.$$
\vglue .25cm
\noindent
(2)  From (1) and integration by parts we have 
\begin{align}
\label{eq80}
&\int_{-\infty}^{a} |\partial_sF_t(s)|(-s)^{1/2+\gamma}ds \leq |H-1/2|t(3/2-H)\int_{-\infty}^{a}(-t-s)^{H-5/2}(-s)^{1/2+\gamma}ds\notag\\ 
&= \;\;|H-1/2|t\left[\left.(-s)^{\gamma+1/2}(-t-s)^{H-3/2}\right|_{-\infty}^{a}+\int_{-\infty}^{a}(1/2+\gamma)(-t-s)^{H-3/2}(-s)^{\gamma-1/2}ds\right].
\end{align}
Since $\gamma<(1-H)\wedge (1/2)$,
\begin{align}
\label{eq81}
\lim_{s\to -\infty}(-s)^{\gamma+1/2}(-t-s)^{H-3/2}&=\lim_{s\to -\infty}(-s)^{\gamma+H-1}\left(\frac{t}{s}+1\right)^{H-3/2}=0,
\end{align}
and
\begin{align*}
&\int_{-\infty}^{a}(-t-s)^{H-3/2}(-s)^{\gamma-1/2}ds\leq \int_{-\infty}^{a}(-t-s)^{H+\gamma-2}ds\notag \\
&=\frac{(-t-a)^{H+\gamma-1}}{1-H-\gamma} <\infty,
\end{align*}
which together (\ref{eq80}) and (\ref{eq81}) shows  that statement $(2)$ holds.
\vglue .25cm
\noindent
(3)  By the pathwise H\"older continuity of $B_3$ on $[1/a, 0]$, taking $0<\gamma< (1-H)\wedge (1/2)$, we have
$|sB(1/s)|< Y(-s)^{1/2-\gamma}$ for each $s\in [1/a, 0]$ and  a random variable $Y$. Then
$|B(s)|< Y(-s)^{1/2+\gamma}$ for each $s\in (-\infty, a]$. Therefore,
\begin{align*}
|F_t(b)B(b)|&\leq \left|(-t-b)^{H-1/2}-(-b)^{H-1/2}\right|Y(-b)^{1/2+\gamma}\notag\\
&=\left|\left(\frac{t}{b}+1\right)^{H-1/2}-1\right|Y(-b)^{H+\gamma},
\end{align*}
and using l'H\^opital rule,
\begin{align*}
&\lim_{b\to -\infty}\frac{\left|\left(\frac{t}{b}+1\right)^{H-1/2}-1\right|}{(-b)^{-\gamma-H}}=0.
\end{align*}
\vglue .25cm
\noindent
(4)   Since $F_t$ is square-integrable on $(-\infty, a)$, $\lim_{b\to -\infty}\int_b^{a} F_t(s)dB(s)=\int_{-\infty}^{a} F_t(s)dB(s)$. Thus, applying integration by parts,
$$\int_b^{a} F_t(s)dB(s)= F_t({a})B({a})- F_t(b)B(b)- \int_b^{a} \partial_sF_t(s)B(s)ds. $$
By  the pathwise H\"older continuity of  $B$ (see the proof of Statement (3)) and (\ref{eqpropF3}),
\begin{align*}
\int_{-\infty}^{a} |\partial_sF_t(s)B(s)|ds < \infty,
\end{align*}
and using (\ref{eqpropF4}),
$$
\int_{-\infty}^{a} F_t(s)dB(s)= F_t(a)B(a)-\int_{-\infty}^{a} \partial_sF_t(s)B(s)ds.
$$
Now, with the change of variable $s=1/v$,
\begin{align*}
\int_{-\infty}^{a} \partial_sF_t(s)B(s)ds&=\int_{1/a}^0\partial_sF_t\left(\frac{1}{v}\right)\frac{1}{v^2}B\left(\frac{1}{v}\right)dv\notag\\
&=\int_{1/a}^0\partial_sF_t\left(\frac{1}{v}\right)\frac{1}{v^3}B_3(v)dv,
\end{align*}
and we obtain (\ref{eq:3.4}).
\end{proof}

We prove Theorem \ref{teoaproxy}  separately for $H>1/2$ and $H<1/2$. 
We denote  the sup norm by  $||\;||_\infty$, and it will always be clear from the context which interval it  refers to.
\vglue .25cm
\noindent 
{\bf 3.1. Case $\mathbf{H>1/2}$}

We fix $H-1/2 <\beta <1/2$ and define
$$
\alpha_n=n^{-(1/2-\beta)}(\log n)^{5/2}.
$$
The proof will be a consequence of the following lemmas, involving $Z^{(n)}_2$ and $Z^{(n)}_3$.

\begin{lemma}
\label{lemB1} For each $q>0$ there is $C>0$ such that  
$$
I_1=P\left(\sup_{0\leq t \leq T}\left| F_t(a)B_2(a) -F_t(a)Z_2^{(n)}(a)\right|> C\alpha_n \right)= o(n^{-q})\ \ \ \text{as} \ \   n\to \infty.
$$
\end{lemma}
\begin{proof}

\begin{align*}
\left| F_t(a)B_2(a) - F_t(a)Z_2^{(n)}(a)\right|&\leq \|B_2 - Z_2^{(n)}\|_{\infty}|(-t-a)^{H-1/2}-(-a)^{H-1/2}|\\
&\leq \|B_2- Z_2^{(n)}\|_{\infty}(-a)^{H-1/2},
\end{align*}
then, by (\ref{eq98}),
\begin{align*}
I_1&\leq P\left(\|B_2 - Z_2^{(n)}\|_{\infty}(-a)^{H-1/2}> C\alpha_n \right)\\
&\leq P\left(\|B_2 - Z_2^{(n)}\|_{\infty}> Cn^{-1/2}(\log n)^{5/2} \right)=o(n^{-q}).
\end{align*}
\end{proof}

\begin{lemma}
\label{lemA1} For each $q>0$ there is $C>0$ such that  
\begin{align*}
I_2&=P\left(\sup_{0\leq t \leq T}\left|\int^0_{-t} (-s)^{H-1/2}dB_2(s) - \int^0_{-t} (-s)^{H-1/2}dZ_2^{(n)}(s)\right|> C\alpha_n  \right)\notag\\
&= o(n^{-q}) \ \ \ \text{as} \ \   n\to \infty.
\end{align*}
\end{lemma}
\begin{proof}
By  integration by parts, 
\begin{align*}
&\int^0_{-t} (-s)^{H-1/2}dB_2(s)=-t^{H-1/2}B_2(-t)+ (H-1/2)\int^0_{-t} 
(-s)^{H-3/2}B_2(s)ds.
\end{align*}
Analogously,
\begin{align*}
\int^0_{-t} (-s)^{H-1/2}dZ_2^{(n)}(s)&=-t^{H-1/2}Z_2^{(n)}(-t)+ (H-1/2)\int^0_{-t} 
(-s)^{H-3/2}Z_2^{(n)}(s)ds,
\end{align*}
then
\begin{align*}
&\left|\int^0_{-t} (-s)^{H-1/2}dB_2(s) - \int^0_{-t} (-s)^{H-1/2}dZ_2^{(n)}(s)\right|\\
&\leq t^{H-1/2}|B_2(-t)-Z_2^{(n)}(-t)|+(H-1/2)\int^0_{-t} (-s)^{H-3/2}\left|B_2(s)-Z_2^{(n)}(s)\right|ds\\
& \leq t^{H-1/2}\|B_2-Z_2^{(n)}\|_{\infty}+ (H-1/2)\|B_2-Z_2^{(n)}\|_{\infty}\left.\frac{-1}{H-1/2}(-s)^{H-1/2}\right|_{-t}^0\\
& \leq 2T^{H-1/2}\|B_2-Z_2^{(n)}\|_{\infty}.
\end{align*}
Consequently the result follows by (\ref{eq98}).
\end{proof}

\begin{lemma}
\label{lemD1} For each $q>0$ there is $C>0$ such that 
$$
I_3=P\left(\sup_{0\leq t \leq T}\left|\int_{a}^{-t} F_t(s)dB_2(s) - \int_{a}^{-t} F_t(s)dZ_2^{(n)}(s)\right|> C\alpha_n \right)= o(n^{-q}) \ \ \ \text{as} \ \   n\to \infty.
$$
\end{lemma}

\begin{proof} By integration by parts,
\begin{align*}
\int_{a}^{-t} F_t(s)dB_2(s)&=F_t(-t)B_2(-t)-F_t(a)B_2(a)- \int_{a}^{-t} \partial_sF_t(s)B_2(s)ds \notag
\end{align*}
and 
\begin{align*}
 \int_{a}^{-t} F_t(s)dZ_2^{(n)}(s)&= F_t(-t)Z_2^{(n)}(-t)-F_t(a)Z_2^{(n)}(a)-\int_{a}^{-t} \partial_sF_t(s)Z_2^{(n)}(s)ds,
\end{align*}
then,
\begin{align*}
&\left|\int_{a}^{-t} F_t(s)dB_2(s) - \int_{a}^{-t}F_t(s)dZ_2^{(n)}(s)\right|\\
&\leq \|B_2-Z^{(n)}_2\|_{\infty}\left[|F_t(-t)|+|F_t(a)| +\int_{a}^{-t}| \partial_sF_t(s)|ds\right]\\
&= \|B_2-Z^{(n)}_2\|_{\infty}\biggl[t^{H-1/2}+ |(-t-a)^{H-1/2}-(-a)^{H-1/2}|\biggr.\\
&\hspace{3.7cm}+\left.\int_{a}^{-t}(H-1/2)[(-t-s)^{H-3/2}-(-s)^{H-3/2}]ds\right]\\
&\leq \|B_2-Z^{(n)}_2\|_{\infty}2T^{H-1/2}.
\end{align*}
Therefore, by  (\ref{eq98}) the proof is complete.
 \end{proof}

\begin{lemma}
\label{lemF1} 
For each $q>0$ there is $C>0$ such that  
\begin{align*}
I_4&=P\biggl(\sup_{0\leq t\leq T}\biggl|\int_{1/a}^{{\varepsilon}_n\vee (1/a)} \partial_sF_t\biggl(\frac{1}{v}\biggr)\frac{1}{v^3}B_3(v)dv \biggr.\biggr.\\
&\hspace{2cm}\biggl.\biggl.- \int_{_{1/a}}^{0}\biggl(-\int_{1/a}^{[{\varepsilon}_n\vee (1/a)]\wedge s}\partial_sF_t\biggl(\frac{1}{v}\biggr)\frac{1}{v^3}dv\biggr)dZ_3^{(n)}(s)\biggr|> C\alpha_n \biggr)\\
&= o(n^{-q}) \ \ \ \text{as} \ \  n\to \infty.
\end{align*}
\end{lemma}

\begin{proof}
We have
\begin{align*}
\int_{1/a}^{{\varepsilon}_n\vee (1/a)} \partial_sF_t\biggl(\frac{1}{v}\biggr)\frac{1}{v^3}B_3(v)dv=& I_{\{\varepsilon_n>1/a\}}\int_{1/a}^{{\varepsilon}_n} \partial_sF_t\biggl(\frac{1}{v}\biggr)\frac{1}{v^3}B_3(v)dv.\\
\end{align*}
Analogously, applying Fubini's theorem we have 

\begin{eqnarray*}
\lefteqn{\int_{_{1/a}}^{0}\biggl(-\int_{1/a}^{[{\varepsilon}_n\vee (1/a)]\wedge r}\partial_sF_t\biggl(\frac{1}{v}\biggr)\frac{1}{v^3}dv\biggr)dZ_3^{(n)}(r)}\\
&=&\int_{1/a}^{{\varepsilon}_n\vee (1/a)} \partial_sF_t\biggl(
\frac{1}{v}\biggr)\frac{1}{v^3}Z_3^{(n)}(v)dv\\
&=&I_{\{\varepsilon_n\geq 1/a\}}\int_{1/a}^{{\varepsilon}_n} 
\partial_sF_t\biggl(\frac{1}{v}\biggr)\frac{1}{v^3}Z_3^{(n)}(v)dv.
\end{eqnarray*}

Then, by (\ref{eqpropF2}),
\begin{eqnarray*}
\lefteqn{\kern .35cm \left|\int_{1/a}^{{\varepsilon}_n\vee (1/a)} \partial_vf_t\left(\frac{1}{v}\right)\frac{1}{v^3}B_3(v)dv - \int_{1/a}^{0}\biggl(-\int_{1/a}^{[{\varepsilon}_n\vee (1/a)]\wedge s} \partial_sF_t\biggl(\frac{1}{v}\biggr)\frac{1}{v^3}dv\biggr)dZ_3^{(n)}(s)\right|}\\
&\leq&\|B_3-Z_3^{(n)}\|_{\infty}I_{\{\varepsilon_n>1/a\}}\int_{1/a}^{{\varepsilon}_n}\left|\partial_sF_t\left(\frac{1}{s}\right)\frac{1}{s^3}\right|ds\\
&\leq&\|B_3-Z_3^{(n)}\|_{\infty}I_{\{\varepsilon_n>1/a\}}\int_{1/a}^{{\varepsilon}_n} t(3/2-H)(H-1/2)(tv+1)^{H-5/2}(-v)^{-1/2-H}dv\\
&\leq& \|B_3-Z_3^{(n)}\|_{\infty}I_{\{\varepsilon_n>1/a\}} t(3/2-H)(H-1/2)\left(t/a+1\right)^{H-5/2}\int_{1/a}^{{\varepsilon_n}} (-v)^{-1/2-H}dv\\
&\leq& \|B_3-Z_3^{(n)}\|_{\infty}I_{\{\varepsilon_n>1/a\}} t(3/2-H)(1+T/a)^{H-5/2}[(-\varepsilon_n)^{1/2-H} -(-1/a)^{1/2-H}]\\
&\leq& \|B_3-Z_3^{(n)}\|_{\infty} T(3/2-H)(1+T/a)^{H-5/2}(-{\varepsilon}_n)^{1/2-H}.
\end{eqnarray*}
Hence, since $(-{\varepsilon}_n)^{1/2-H}=n^{\beta}$, by (\ref{eqdefepsilon}),
\begin{align*}
I_4&\leq P\left(\|B_3-Z_3^{(n)}\|_{\infty} T(3/2-H)(1+T/a)^{H-5/2}(-{\varepsilon}_n)^{1/2-H}>C\alpha_n\right)\notag\\
&\leq P\left(\|B_3-Z_3^{(n)}\|_{\infty} >Cn^{-\beta}\alpha_n\right)\notag\\
&\leq P\left(\|B_3-Z_3^{(n)}\|_{\infty} >Cn^{-1/2}(\log n)^{5/2}\right)=o(n^{-q}).
\end{align*}
\end{proof}

\begin{lemma}
\label{lemG1} For each $q>0$,  
$$
I_5=P\left(\sup_{0\leq t \leq T}\left|\int_{{\varepsilon}_n\vee (1/a)}^0 \partial_sF_t\left(\frac{1}{v}\right)\frac{1}{v^3}B_3(v)dv \right|> \alpha_n\right)=o(n^{-q})\ \ \ \text{as} \ \  n\to \infty.
$$
\end{lemma}

\begin{proof} By the pathwise H\"older continuity of  $B_3$ with  $0<\gamma<1-H$, and  (\ref{eqpropF2}),   
\begin{align*}
&\kern -.35cm \left|\int_{{\varepsilon}_n\vee (1/a)}^0\partial_sF_t\left(\frac{1}{v}\right)\frac{1}{v^3}B_3(v)dv\right|\\
\leq &\int_{{\varepsilon}_n\vee (1/a)}^0t(3/2-H)(H-1/2)(tv+1)^{H-5/2}(-v)^{-H-1/2}Y(-v)^{1/2-\gamma}dv\\
\leq &\;(3/2-H)(H-1/2)TY (1+T/a)^{H-5/2}
\int_{{\varepsilon}_n\vee (1/a)}^0(-v)^{-H-\gamma}dv\\
\leq & \;CY({-\varepsilon}_n)^{1-H-\gamma}\\
= & \;CYn^{-\beta(1-H-\gamma)/(H-1/2)},
\end{align*}
where $C$ is a positive constant.

By Chebyshev's inequality, for $r>0$,
\begin{align*}
I_5&\leq P\left(CYn^{-\beta(1-H-\gamma)/(H-1/2)}>\alpha_n\right)\\
&=P\left(CY>n^{\kappa}(\log n)^{5/2}\right)\\
&\leq \frac{E(|CY|^r)}{n^{r\kappa}(\log n)^{r5/2}},
\end{align*}
where $\kappa=-(1/2-\beta)+ \beta(1-H-\gamma)/(H-1/2)$. Taking $\gamma$ close enough to $0$ we have $H-1/2<(H-1/2)/(1-2\gamma)<\beta<1/2$, and then $\kappa>0$. For $q>0$ there is $r>0$ such that $q<r\kappa$, then
$$
	\lim_{n\to \infty}n^qI_5=0. 
$$
\end{proof}

\begin{proof}[\textbf{Proof of Theorem \ref{teoaproxy} for}  $\mathbf{H>1/2}$:]

From (\ref{eqdefY}), (\ref{eqdeffunF}) and (\ref{eq:3.4})  we have
\begin{align*}
    Y(t)&=C\left\{-\int_{-t}^0 (-s)^{H-1/2}dB(s) + \int_{a}^{-t} F_t(s)dB(s) + \int_{-\infty}^{a} F_t(s)dB(s)\right\}\\
    &=C\left\{ -\int_{-t}^0(-s)^{H-1/2}dB_2(s) + \int_{a}^{-t} F_t(s)dB_2(s)+F_t(a)B_2(a)\right.\\
    &\hspace{1cm} - \left.\int_{1/a}^0 \partial_sF_t\left(\frac{1}{v}\right)\frac{1}{v^3}B_3(v)dv\right\}\\
&=C\left\{-\int_{-t}^0 (-s)^{H-1/2}dB_2(s) + \int_{a}^{-t} F_t(s)dB_2(s)+F_t(a)B_2(a)\right. \\
    &\hspace{1cm}\left.- \int_{1/a}^{\varepsilon_n\vee(1/a)} \partial_sF_t\left(\frac{1}{v}\right)\frac{1}{v^3}B_3(v)dv - \int_{\varepsilon_n\vee(1/a)}^0 \partial_sF_t\left(\frac{1}{v}\right)\frac{1}{v^3}B_3(v)dv\right\},
\end{align*}
then the definition of $Y^{(n)}$ (see (\ref{eqdefyn})) implies
\begin{align*}
    &|Y(t)-Y^{(n)}(t)|\leq C\Biggl\{\Biggl.\left|-\int_{-t}^0 (-s)^{H-1/2}dB_2(s)+\int_{-t}^0 (-s)^{H-1/2}dZ^{(n)}_2(s)\right|\\
    &+\left|\int_{a}^{-t} F_t(s)dB_2(s)-\int_{a}^{-t} F_t(s)dZ^{(n)}_2(s)\right|+\left|F_t(a)B_2(a)-F_t(a)Z_2^{(n)}(a)\right|\\
    &+\left|\int_{1/a}^{\varepsilon_n\vee(1/a)} \partial_sF_t\left(\frac{1}{v}\right)\frac{1}{v^3}B_3(v)dv-\int_{_{1/a}}^{0}\biggl(-\int_{1/a}^{[{\varepsilon}_n\vee (1/a)]\wedge s}\partial_sF_t\biggl(\frac{1}{v}\biggr)\frac{1}{v^3}dv\biggr)dZ_3^{(n)}(s)\right|\\
    &+\left|\int_{{\varepsilon}_n\vee (1/a)}^0 \partial_sF_t\left(\frac{1}{v}\right)\frac{1}{v^3}B_3(v)dv \right|\Biggl.\Biggr\}.
\end{align*}
Therefore, taking $\beta$ such that $0<H-1/2<\beta<1/2$,  by  Lemmas \ref{lemB1}, \ref{lemA1}, \ref{lemD1}, \ref{lemF1} and \ref{lemG1} we have the result.
\end{proof}

\noindent
{\bf 3.2. Case $\mathbf{H<1/2}$}

Let $1/2-H<\beta <1/2$, and  $\varepsilon_n$ and $\alpha_n$ are  as before. We proceed similarly with some lemmas.

\begin{lemma}
\label{lemA32}
For each $q>0$ there is $C$ such that
$$
J_1=P\left(\sup_{0\leq t \leq T}\left| F_t(a)B_2(a)-F_t(a)Z_2^{(n)}(a)\right|> C\alpha_n \right)= o(n^{-q})\ \ \ \text{as} \ \  n\to \infty.
$$
\end{lemma}

\begin{proof}
Similar arguments as in the proof of Lemma \ref{lemB1}.
\end{proof}

\begin{lemma}
\label{lemA22} For each $q>0$ there is $C>0$ such that 
\begin{eqnarray*}
\label{eqA22} J_2&=&P\left(\sup_{0\leq t \leq
T}\left|
\int_{-t}^{\varepsilon_n\vee (-t)}(-s)^{H-1/2}dB_2(s) -
\int_{-t}^{\varepsilon_n\vee (-t)}(-s)^{H-1/2}dZ_2^{(n)}(s)\right|>C\alpha_n\right)\notag\\
&=& o(n^{-q})\ \ \ \text{as} \ \  n\to \infty.
\end{eqnarray*}
\end{lemma}

\begin{proof}
By integration by parts,
\begin{eqnarray*}
\int_{-t}^{\varepsilon_n\vee (-t)}(-s)^{H-1/2}dB_2(s)&=&I_{\{\varepsilon_n> -t\}}\biggl[\int_{-t}^{\varepsilon_n}(-s)^{H-1/2}dB_2(s)\biggr]\notag\\
&=&I_{\{\varepsilon_n> -t\}}\biggl[(-\varepsilon_n)^{H-1/2}B_2(\varepsilon_n)-t^{H-1/2}B_2(-t)\biggr.\notag\\
 &&\hspace{1.5cm}\biggl.  +\int_{-t}^{\varepsilon_n}(H-1/2)(-s)^{H-3/2}B_2(s)ds\biggr],
\end{eqnarray*}

and analogously, 
\begin{align*}
&\int_{-t}^{\varepsilon_n\vee (-t)}(-s)^{H-1/2}dZ_2^{(n)}(s)=I_{\{\varepsilon_n> -t\}}\biggl[-(\varepsilon_n)^{H-1/2}Z_2^{(n)}(\varepsilon_n) -t^{H-1/2}Z_2^{(n)}(-t)\biggr.\notag\\
&\biggl.\hspace{6.7cm}+\int_{-t}^{\varepsilon_n}(H-1/2)(-s)^{H-3/2}Z_2^{(n)}(s)ds\biggr]\notag .
\end{align*}
We have
\begin{align*}
&\left|\int_{-t}^{\varepsilon_n\vee (-t)}(-s)^{H-1/2}dB_2(s) - \int_{-t}^{\varepsilon_n\vee (-t)}(-s)^{H-1/2}dZ_2^{(n)}(s)\right|\notag\\
&\leq I_{\{\varepsilon_n> -t\}}\|B_2-Z^{(n)}_2\|_{\infty}\biggl[(-\varepsilon_n)^{H-1/2}+ t^{H-1/2}+\int_{-t}^{\varepsilon_n}(1/2-H)(-s)^{H-3/2}ds\biggr]\notag\\
&\leq 2\|B_2-Z^{(n)}_2\|_{\infty}{n^{\beta}}.\notag 
\end{align*}
Then,
\begin{align*}
J_2&\leq P\left(2\|B_2-Z^{(n)}_2\|_{\infty}n^{\beta} >C\alpha_n\right)\leq P\left(\|B_2-Z^{(n)}_2\|_{\infty} >Cn^{-1/2}(\log n)^{5/2}\right)=o(n^{-q}).\notag
\end{align*}
\end{proof}

\begin{lemma}
\label{lemD22} For $1/2-H< \beta < 1/2$ and each $q>0$, 
\begin{eqnarray*}
\label{eqD22}
J_3&=&P\left(\sup_{0\leq t \leq T}\left|\int_{\varepsilon_n\vee (-t)}^0 [(-s)^{H-1/2}- (-s-\varepsilon_n)^{H-1/2}]dB_2(s) \right|> \alpha_n \right)\notag\\
&=&o(n^{-q}) \ \ \ \text{as} \ \  n\to \infty.
\end{eqnarray*}
\end{lemma}

\begin{proof} By the H\"older continuity of  $B_2$  with $0<\gamma<H$,
\begin{align}
\label{eqD22-1}
&\left|\int_{\varepsilon_n\vee (-t)}^0 [(-s)^{H-1/2}- (-s-\varepsilon_n)^{H-1/2}]dB_2(s)\right|=\left|\int_{\varepsilon_n\vee (-t)}^0 \int_{-s}^{-s-\varepsilon_n} (1/2-H)x^{H-3/2}dxdB_2(s)\right|\notag\\ 
&=\left|\int_0^{-\varepsilon_n-(\varepsilon_n\vee(-t))}\int_{(-t\vee\varepsilon_n)\vee (-x)}^{(-x-\varepsilon_n)\wedge 0}(1/2-H)x^{H-3/2}dB_2(s)dx\right|\notag\\
&=\left|\int_0^{-\varepsilon_n-(\varepsilon_n\vee(-t))}(1/2-H)x^{H-3/2}[B_2((-x-\varepsilon_n)\wedge 0)-B_2((-t\vee\varepsilon_n)\vee (-x))]dx\right|\notag\\
&\leq (1/2-H)Y\int_0^{-\varepsilon_n-(\varepsilon_n\vee(-t))}x^{H-3/2}A_1^{1/2-\gamma}(x)dx,
\end{align}
where 
$$
	A_1(x)=|(-x-\varepsilon_n)\wedge 0- ((-t)\vee\varepsilon_n\vee (-x))|.
$$
First, if $0\leq -x-\varepsilon_n$, then $A_1(x)=| (-t)\vee(-x)|$, and if $t<x$, then
\begin{equation}
\label{eqA11}
	A_1(x)=t<x<2x.
\end{equation}
If $t\geq x$, then
\begin{equation}
\label{eqA12}
	A_1(x)=x<2x.
\end{equation}
Second, if $0> -x-\varepsilon_n$, then 
\begin{eqnarray*}
A_1(x)&=&|-x-\varepsilon_n- ((-t)\vee \varepsilon_n\vee(-x))|=	|-x-\varepsilon_n- ((-t)\vee \varepsilon_n)|\\
&=&-x-\varepsilon_n- ((-t)\vee \varepsilon_n) \leq -\varepsilon_n+(t\wedge(-\varepsilon_n)).
\end{eqnarray*}
 If $t<-\varepsilon_n$, then
\begin{equation}
\label{eqA13}
	A_1(x)\leq -\varepsilon_n+t\leq -2\varepsilon_n<2x,
\end{equation}
  and if $t\geq -\varepsilon_n$, then
\begin{equation}
\label{eqA14}
	A_1(x)= -\varepsilon_n-\varepsilon_n\leq -2\varepsilon_n<2x.
\end{equation}
From (\ref{eqA11})-(\ref{eqA14}) we have that $A_1(x)\leq 2x$, and then from (\ref{eqD22-1}),
\begin{align*}
	&\kern -.26cm\left|\int_{\varepsilon_n\vee (-t)}^0 [(-s)^{H-1/2}- (-s-\varepsilon_n)^{H-1/2}]dB_2(s)\right|\\
\leq&\;\; (1/2-H)2^{1/2-\gamma}Y\int_0^{-\varepsilon_n-(\varepsilon_n\vee(-t))}x^{H-1-\gamma}dx\\
	=&\;\; \frac{(1/2-H)2^{1/2-\gamma}}{H-\gamma}Y({-\varepsilon_n-(\varepsilon_n\vee(-t))})^{H-\gamma}\\\
	\leq&\;\; \left.\frac{(1/2-H)2^{H-2\gamma+1/2}}{H-\gamma}Y
({-\varepsilon_n})^{H-\gamma}.\right.
	\end{align*}
Hence
\begin{align*}
J_3&\leq P\left(\frac{(1/2-H)2^{H-2\gamma+1/2}}{H-\gamma}Y({-\varepsilon_n})^{H-\gamma}>\alpha_n\right)\\
&=\;\;P\left(CY>({-\varepsilon_n})^{-H+\gamma}\alpha_n\right)\\
&=\;\;P\left(CY>n^{\kappa}(\log n)^{5/2}\right)\\
&\leq \;\;\frac{E(\left|CY\right|^r)}{n^{r\kappa}(\log n)^{r5/2}},
\end{align*}
where  $\kappa=-(1/2-\beta)- \beta(H-\gamma)/(H-1/2)$. Taking  $\gamma$ close enough to $0$ we have $0<(1/2-H)/(1-2\gamma)<\beta<1/2$, and then $\kappa>0$. 
The result follows by analogous arguments as in proof of the Lemma \ref{lemG1}. 
\end{proof}

\begin{lemma}
\label{lemF22} For each $q>0$ there is $C$ such that

\begin{align*}
J_4&= P\left(\sup_{0\leq t \leq T}\left|
\int^0_{\varepsilon_n\vee (-t)}(-s-\varepsilon_n)^{H-1/2}dB_2(s)-\int^0_{\varepsilon_n\vee (-t)}(-s-\varepsilon_n)^{H-1/2}dZ_2^{(n)}(s)\right|>C\alpha_n \right)\notag\\
& = o(n^{-q}) \ \ \ \text{as} \ \  n\to \infty.
\end{align*}

\end{lemma}

\begin{proof} By integration by parts
\begin{align*}
&\int^0_{\varepsilon_n\vee (-t)}(-s-\varepsilon_n)^{H-1/2}dB_2(s)\\
&=-(-(\varepsilon_n\vee (-t))-\varepsilon_n)^{H-1/2}B_2(\varepsilon_n\vee (-t))
+\int^0_{\varepsilon_n\vee (-t)}(H-1/2)(-s-\varepsilon_n)^{H-3/2}B_2(s)ds,
\end{align*}
and

\begin{align*}
 &\kern -.29cm \int^0_{\varepsilon_n\vee (-t)}(-s-\varepsilon_n)^{H-1/2}dZ_2^{(n)}(s)\notag\\
=&-(-(\varepsilon_n\vee (-t))-\varepsilon_n)^{H-1/2}Z_2^{(n)}(\varepsilon_n\vee (-t))
+\int^0_{\varepsilon_n\vee (-t)}(H-1/2)(-s-\varepsilon_n)^{H-3/2}Z_2^{(n)}(s)ds.
\end{align*}
Then 
\begin{align*}
&\left|
\int^0_{\varepsilon_n\vee (-t)}(-s-\varepsilon_n)^{H-1/2}dB_2(s)-\int^0_{\varepsilon_n\vee (-t)}(-s-\varepsilon_n)^{H-1/2}dZ_2^{(n)}(s)\right|\notag\\
&\leq\|B_2-Z^{(n)}_2\|_{\infty}\biggl[(-(\varepsilon_n\vee (-t))-\varepsilon_n)^{H-1/2}+ \int^0_{\varepsilon_n\vee (-t)}(1/2-H)(-s-\varepsilon_n)^{H-3/2}ds \biggr]\notag\\
&=\|B_2-Z^{(n)}_2\|_{\infty}(-\varepsilon_n)^{H-1/2}\notag\\
&=\|B_2-Z^{(n)}_2\|_{\infty}n^\beta.
\end{align*}
Finally,
\begin{align*}
J_4\leq &\;P\left(\|B_2-Z^{(n)}_2\|_{\infty}n^{\beta}>C\alpha_n \right)=P\left(\|B_2-Z^{(n)}_2\|_{\infty}> Cn^{-1/2}(\log n)^{5/2} \right)\\
&\kern -.33cm =o(n^{-q}).
\end{align*}
\end{proof}

\begin{lemma}
\label{lemG22} For each $q>0$ there is $C$ such that 
\begin{align*}
J_5&=P\left(\sup_{0\leq t \leq T}\right.\left|\int^0_{1/a} \partial_sF_t\left(\frac{1}{v}\right)\frac{1}{v^3}B_3(v)dv \right.\\ 
&\hspace{2.3cm}- \left.\left.\int^0_{1/a}\left(-\int_{{1/a}}^{s}\partial_sF_t\left(\frac{1}{v}\right)\frac{1}{v^3}dv\right)dZ_3^{(n)}(s)\right|> C\alpha_n \right)\\
&= o(n^{-q}) \ \ \ \text{as} \ \  n\to \infty.
\end{align*}
\end{lemma}

\begin{proof}By the Fubini's theorem we have
$$\int^0_{1/a}\left(-\int_{1/a}^{s}\partial_sF_t\left(\frac{1}{v}\right)\frac{1}{v^3}dv\right)dZ_3^{(n)}(s)=\int^0_{1/a} \partial_sF_t\left(\frac{1}{v}\right)\frac{1}{v^3}Z^{(n)}_3(v)dv,$$ 
then, by Lemma \ref{lemprop}, 
\begin{align*}
&\left|\int^0_{1/a} \partial_sF_t\left(\frac{1}{v}\right)\frac{1}{v^3}B_3(v)dv - \int^0_{1/a}\left(-\int_{{1/a}}^{s}\partial_sF_t\left(\frac{1}{v}\right)\frac{1}{v^3}dv\right)dZ_3^{(n)}(s)\right|\\
&\leq \|B_3-Z_3^{(n)}\|_{\infty}\int^0_{1/a} \left|\partial_sF_t
\left(\frac{1}{v}\right)\frac{1}{v^3}\right|dv\\
&\leq \|B_3-Z_3^{(n)}\|_{\infty}\int^0_{1/a} t(3/2-H)(1/2-H)(tv+1)^{H-5/2}(-v)^{-1/2-H}dv \\
&\leq \|B_3-Z_3^{(n)}\|_{\infty}t(3/2-H)(t/a+1)^{H-5/2}(1/2-H)\int^0_{1/a} (-v)^{-1/2-H}dv \\
&\leq \|B_3-Z_3^{(n)}\|_{\infty}T(3/2-H)(T/a+1)^{H-5/2}(-1/a)^{1/2-H}. \\
\end{align*}
Therefore,
$$
J_5 \leq  P\left(\|B_3-Z_3^{(n)}\|_{\infty}>C\alpha_n \right)=o(n^{-q}).
$$
\end{proof}

\begin{lemma}
\label{lemJ22} For each $q>0$ there is $C$ such that 
\begin{align*}
J_6&=P\left(\sup_{0\leq t \leq T}\left|\int^{a\vee(-t+\varepsilon_n)}_{a} F_t(s)dB_2(s) -\int^{a\vee(-t+\varepsilon_n)}_{a} F_t(s)dZ_2^{(n)}(s)\right|> C\alpha_n \right)\\
&= o(n^{-q}) \ \ \ \text{as} \ \  n\to \infty.
\end{align*}
\end{lemma}

\begin{proof}
By integration by parts,

\begin{align*}
\int^{a\vee(-t+\varepsilon_n)}_{a} F_t(s)dB_2(s)
&=F_t(a\vee(-t+\varepsilon_n))B_2(a\vee(-t+\varepsilon_n))-F_t(a)B_2(a)\\
&\ \ \ \ -\int^{a\vee(-t+\varepsilon_n)}_{a} \partial_sF_t(s)B_2(s)ds,
\end{align*}
and
\begin{align*}
\int^{a\vee(-t+\varepsilon_n)}_{a} F_t(s)dZ_2^{(n)}(s)&=F_t(a\vee(-t+\varepsilon_n))Z_2^{(n)}(a\vee(-t+\varepsilon_n))\\
&\ \ \ \ -F_t(a)Z_2^{(n)}(a) -\int^{a\vee(-t+\varepsilon_n)}_{a} \partial_sF_t(s)Z_2^{(n)}(s)ds.
\end{align*}
Then
\begin{align*}
&\kern -.3cm\left|\int^{a\vee(-t+\varepsilon_n)}_{a} F_t(s)dB_2(s)-\int^{a\vee(-t+\varepsilon_n)}_{a} F_t(s)dZ_2^{(n)}(s)\right|\\
\leq&\;\;  \|B_2-Z^{(n)}_2\|_{\infty} \left[|F_t(a\vee(-t+\varepsilon_n))|+
|F_t(a)|+\int^{a\vee(-t+\varepsilon_n)}_{a}| \partial_sF_t(s)|ds\right]\\
=&\;\; \|B_2-Z^{(n)}_2\|_{\infty}
\biggl[(-t-(a\vee(-t+\varepsilon_n)))^{H-1/2}-(-(a\vee(-t+\varepsilon_n)))^{H-1/2}\biggr.\\
&\biggl.+(-t-a)^{H-1/2}-(-a)^{H-1/2}+\int^{a\vee(-t+\varepsilon_n)}_{a}(1/2-H)[(-t-s)^{H-3/2}-(-s)^{H-3/2}]ds\biggr]\\
=&\;\;  \|B_2-Z^{(n)}_2\|_{\infty}
2\biggl[(-t-(a\vee(-t+\varepsilon_n)))^{H-1/2}-(-(a\vee(-t+\varepsilon_n)))^{H-1/2}\biggr]\\
\leq&\;\;\|B_2-Z^{(n)}_2\|_{\infty}2((-\varepsilon_n)^{H-1/2}+(-T-a)^{H-1/2}).
\end{align*}
Hence the result follows.
\end{proof}

\begin{lemma}
\label{lemJ23} For each $q>0$ there is $C$ such that 
\begin{align*}
J_7&=P\left(\sup_{0\leq t \leq T}\left|I_{\left\{-\varepsilon_n\leq t \right\}}
\int^{-t}_{a\vee (-t+\varepsilon_n)} F_{t+\varepsilon_n}(s)dB_2(s)\right.\right.\\
&\hspace{2cm}\left.\left.-I_{\left\{-\varepsilon_n\leq t \right\}}
\int^{-t}_{a\vee (-t+\varepsilon_n)} F_{t+\varepsilon_n}(s)dZ_2^{(n)}(s)\right|> C\alpha_n \right)\\
&= o(n^{-q}) \ \ \ \text{as} \ \  n\to \infty.
\end{align*}
\end{lemma}

\noindent
{\it Proof.} By integration by parts,
\vskip.10cm
\begin{align*}
&I_{\left\{-\varepsilon_n\leq t \right\}}\left|
\int^{-t}_{a\vee (-t +\varepsilon_n)} F_{t+\varepsilon_n}(s)dB_2(s)-
\int^{-t}_{a\vee (-t+\varepsilon_n)} F_{t+\varepsilon_n}(s)dZ_2^{(n)}(s)\right|\\
&=\;\;I_{\{-\varepsilon_n \leq t\}}\left|F_{t+\varepsilon_n}(-t) (B_2 (-t)-Z^{(n)}_2 (-t))
\right.\\
&-F_{t+\varepsilon_n} (a\vee (-t+\varepsilon_n))(B_2(a\vee (-t+\varepsilon_n))-Z^{(n)}_2 
(a\vee (-t+\varepsilon_n))\\
&\left.-\int^{-t}_{a\vee (-t+\varepsilon_n)} \partial_s F_{t+\varepsilon_n}(s)(B_2(s)-Z^{(n)}_2 (s))ds\right|\\
&\leq \;\; I_{\{-\varepsilon_n \leq t\}}||B_2 -Z^{(n)}_2 ||_\infty 
\biggl(|F_{t+\varepsilon_n} (-t)|+|F_{t+\varepsilon_n}(a\vee (-t+\varepsilon_n))|\biggr. \left.+ \int^{-t}_{a\vee(-t+\varepsilon_n)} |\partial_s F_{t+\varepsilon_n} (s) |ds\right)
\end{align*}
\begin{align*}
&=\;\;I_{\{-\varepsilon_n \leq t\}} ||B_2 -Z^{(n)}_2 ||_\infty 
\biggl((-\varepsilon_n)^{H-1/2}-(t)^{H-1/2} +(-t-\varepsilon_n-(a\vee (-t+\varepsilon_n)))^{H-1/2}\biggr.\\
&-\;\left. (-(a\vee (-t+\varepsilon_n)))^{H-1/2}+\left(1/2-H\right) \int^{-t}_{a\vee (-t+\varepsilon_n)} [(-s-t-\varepsilon_n)^{H-3/2}-(-s)^{H-3/2}]ds\right)\\
&=\; I_{\{-\varepsilon_n \leq t\}} ||B_2-Z^{(n)}_2 ||_\infty 2\left((-\varepsilon_n)^{H-1/2}-t^{H-1/2}\right)\\
&\leq\;\;2\|B_2-Z^{(n)}_2\|_{\infty}(-\varepsilon_n)^{H-1/2},
\end{align*}
and we have the result similarly as the Lemma \ref{lemF22}. \hfill $\Box$

\begin{lemma}
\label{lemK22} For $1/2-H<\beta<1/2$ and each $q>0$,
\begin{align*}
&J_8=P\left(\sup_{0\leq t \leq
T}\left|I_{\left\{t< -\varepsilon_n \right\}}\int_{a\vee(-t+\varepsilon_n)}^{-t} F_t(s)dB_2(s) \right|> C\alpha_n \right)=
o(n^{-q}) \ \ \ \text{as} \ \  n\to \infty.
\end{align*}
\end{lemma}

\begin{proof} 
By the H\"older continuity of  $B_2$ with $0<\gamma <H$,
\begin{align}
\label{eq152}
&\kern -.45cm \left|I_{\left\{t< -\varepsilon_n \right\}}\int_{a\vee(-t+\varepsilon_n)}^{-t} F_t(s)dB_2(s)\right|\notag\\
=&\;\;\left|I_{\left\{t< -\varepsilon_n \right\}}\int_{a\vee(-t+\varepsilon_n)}^{-t}\int_{-t-s}^{-s}(1/2-H)x^{H-3/2}dxdB_2(s)\right|\notag\\
=&\;\;\left|I_{\left\{t< -\varepsilon_n \right\}}\int_0^{-(a\vee(-t+\varepsilon_n))}\int_{(-t-x)\vee a\vee(-t+\varepsilon_n)}^{(-x)\wedge(-t)}(1/2-H)x^{H-3/2}dB_2(s)dx\right|\notag\\
=&\;\;\biggl|I_{\left\{t< -\varepsilon_n \right\}}\int_0^{-(a\vee(-t+\varepsilon_n))}(1/2-H)x^{H-3/2}[B_2((-x)\wedge(-t))\biggr.\notag\\
&\hspace{7cm}\biggl.-B_2((-t-x)\vee a\vee(-t+\varepsilon_n))]dx\biggr|\notag\\
\leq&\;\;(1/2-H)YI_{\left\{t< -\varepsilon_n \right\}}\int_0^{-(a\vee(-t+\varepsilon_n))}x^{H-3/2}(A_2(x))^{1/2-\gamma}dx
\end{align}
where
$$
A_2(x)=|((-x)\wedge(-t))-((-t-x)\vee a\vee(-t+\varepsilon_n))|.
$$
First, if $-x<-t$ and $-t-x<-t+\varepsilon_n$, then
\begin{equation}
\label{eqA2-1}
A_2(x)=-x+(-a\wedge (t-\varepsilon_n))<t<x.
\end{equation}
If $-x<-t$ and $-t-x\geq-t+\varepsilon_n$, then
\begin{equation}
\label{eqA2-2}
A_2(x)=-x+((t+x)\wedge (-a))<t<x.
\end{equation}
Second, if $-x\geq-t$ and $-t-x<-t+\varepsilon_n$, then
\begin{equation}
\label{eqA2-3}
A_2(x)=-t+(-a\wedge (t-\varepsilon_n))<-\varepsilon_n<x.
\end{equation}
If $-x\geq -t$ and $-t-x\geq-t+\varepsilon_n$, then
\begin{equation}
\label{eqA2-4}
A_2(x)=-t+((t+x)\wedge (-a))<x.
\end{equation}

From (\ref{eqA2-1})-(\ref{eqA2-4}) we have that $A_2(x)\leq x$, then by (\ref{eq152}),
\begin{align*}
\left|I_{\left\{t< -\varepsilon_n \right\}}\int_{a\vee(-t+\varepsilon_n)}^{-t} 
F_t(s)dB_2(s)\right|
\leq &\;\;(1/2-H)YI_{\left\{t< -\varepsilon_n \right\}}
\int_0^{-(a\vee (-t+\varepsilon_n))}x^{H-\gamma-1}dx\\
\leq &\;\;\frac{1/2-H}{H-\gamma}YI_{\left\{t< -\varepsilon_n \right\}}
({t-\varepsilon_n})^{H-\gamma}\\
\leq&\;\;\frac{1/2-H}{H-\gamma}2^{H-\gamma}Y(-\varepsilon_n)^{H-\gamma}.
\end{align*}
Proceeding similary as in Lemma \ref{lemD22} we have the result.
\end{proof}

\begin{lemma}
\label{lemL22} For $1/2-H<\beta<1/2$ and each $q>0$,
\begin{align*}
&J_9=P\left(\sup_{0\leq t \leq
T}\left|I_{\left\{-\varepsilon_n \leq t\right\}}\int_{a\vee(-t+\varepsilon_n)}^{-t} [F_t(s)-F_{t+\varepsilon_n}(s)]dB_2(s) \right|> C\alpha_n \right)\nonumber\\
&\;=o(n^{-q}) \ \ \ \text{as} \ \  n\to \infty.
\end{align*}
\end{lemma}

\begin{proof} 
By the H\"older continuity of  $B_2$ with $0<\gamma <H$,
\begin{align}
\label{eq153}
&\kern -.35cm\left|I_{\left\{-\varepsilon_n \leq t \right\}}\int_{a\vee(-t+\varepsilon_n)}^{-t} [F_t(s)-F_{t+\varepsilon_n}(s)]dB_2(s)\right|\notag\\
=&\;\;\left|I_{\left\{-\varepsilon_n \leq t \right\}}\int_{a\vee(-t+\varepsilon_n)}^{-t}\int_{-t-s}^{-t-s-\varepsilon_n}(1/2-H)x^{H-3/2}dxdB_2(s)\right|\notag\\
=&\;\;\biggl|I_{\left\{-\varepsilon_n \leq t \right\}}\int_0^{-t-\varepsilon_n+((-a)\wedge(t-\varepsilon_n))}(1/2-H)x^{H-3/2}[B_2((-t-x-\varepsilon_n)\wedge(-t))\biggr.\notag\\
&\hspace{8cm}\biggl.-B_2((-t-x)\vee a\vee(-t+\varepsilon_n))]dx\biggr|\notag\\
\leq&\;\;(1/2-H)YI_{\left\{-\varepsilon_n \leq t \right\}}\int_0^{-t-\varepsilon_n+((-a)\wedge(t-\varepsilon_n))}x^{H-3/2}(A_3(x))^{1/2-\gamma}dx,
\end{align}
where
$$
A_3(x)=|(-t-x-\varepsilon_n)\wedge (-t)-((-t-x)\vee (-t+\varepsilon_n)\vee a)|\leq x.
$$
Then, by (\ref{eq153}),
\begin{align*}
&\kern -.35cm \left|I_{\left\{-\varepsilon_n \leq t \right\}}\int_{a\vee(-t+\varepsilon_n)}^{-t} [F_t(s)-F_{t+\varepsilon_n}(s)]dB_2(s)\right|\notag\\
\leq&\;\;(1/2-H)YI_{\left\{-\varepsilon_n \leq t \right\}}\int_0^{-t-\varepsilon_n+((-a)\wedge(t-\varepsilon_n))}x^{H-1-\gamma}dx\notag\\
\leq&\;\;Y\frac{1/2-H}{H-\gamma}(-t-\varepsilon_n+((-a)\wedge(t-\varepsilon_n)))^{H-\gamma}\notag\\
\leq&\;\;Y\frac{1/2-H}{H-\gamma}2^{H-\gamma}(-\varepsilon_n)^{H-\gamma}.
\end{align*}
Proceeding similary as in Lemma \ref{lemD22} we have the result.
\end{proof}

\begin{proof}[\textbf{Proof of Theorem \ref{teoaproxy} for}  $\mathbf{H<1/2}$:]

From (\ref{eqdefY}), (\ref{eqdeffunF}) and (\ref{eq:3.4}) we obtain
\begin{align*}
    Y(t)=&\;C\left\{-\int^0_{-t} (-s)^{H-1/2}dB(s) + \int_{a}^{-t} F_t(s)dB(s) + \int_{-\infty}^{a} F_t(s)dB(s)\right\}\\
    =&\;C\left\{-\int^{\varepsilon_n\vee (-t)}_{-t} (-s)^{H-1/2}dB_2(s) -  \int_{\varepsilon_n\vee(-t)}^0 [(-s)^{H-1/2}-(-s-\varepsilon_n)^{H-1/2}]dB_2(s)\right.\\
    &-\int_{\varepsilon_n\vee(-t)}^0 (-s-\varepsilon_n)^{H-1/2}dB_2(s)+\int_{a}^{a\vee(-t+\varepsilon_n)} F_t(s)dB_2(s)\\
    & +I_{\left\{t< -\varepsilon_n \right\}}\int_{a\vee(-t+\varepsilon_n)}^{-t} F_t(s)dB_2(s)+I_{\left\{ -\varepsilon_n \leq t \right\}}\int_{a\vee(-t+\varepsilon_n)}^{-t} [F_t(s)-F_{t+\varepsilon_n}(s)]dB_2(s)\\
&+I_{\left\{ -\varepsilon_n \leq t \right\}}\int_{a\vee(-t+\varepsilon_n)}^{-t} F_{t+\varepsilon_n}(s)dB_2(s)+F_t(a)B_2(a)\\
&\left.-\int_{1/a}^{0}\partial_sF_t\left(\frac{1}{v}\right)\frac{1}{v^3}B_3(v)dv\right\},
\end{align*}
and we have the result similarly as the case $H>1/2$.
\end{proof} 
\noindent
{\bf Acknowledgment}
\vglue .25cm
This work was done with support of CONACyT grant 98998.

\noindent
Johanna Garz\'on\\
Department of Statistics\\
University of Valpara\'{\i}so, Chile\\
{\tt margaret.garzon@uv.cl}\\[.5cm]
Luis G. Gorostiza\\
Department of Mathematics\\
CINVESTAV-IPN, Mexico\\
{\tt lgorosti@math.cinvestav.mx}\\[.5cm]
Jorge A. Le\'on\\
Department of Automatic Control\\
CINVESTAV-IPN, Mexico\\
{\tt jleon@ctrl.cinvestav.mx}
\end{document}